\documentclass{article}
\usepackage[utf8]{inputenc}

\usepackage{graphicx,graphics}

\usepackage{rotating}
\usepackage[twoside]{geometry}
\usepackage{epsfig}
\usepackage{cmap}

\usepackage{graphicx,graphics}

\usepackage{hyperref}
\hypersetup{
pdftitle={Long-rangediscreteDirac},            
colorlinks=true,      
linkcolor=black,         
citecolor=blue}

\usepackage{amssymb,amsfonts,amstext,amsthm}
\usepackage{mathrsfs}
\usepackage[intlimits]{amsmath}
\newtheorem{proposition}{\bf Proposition}[section]

\newtheorem{lemma}{\bf Lemma}[section]
\newtheorem{theorem}{\bf Theorem}[section]

\newtheorem{remark}{Remark}

\newcounter{SE}

\title{Localization and eigenvalue asymptotics for long-range discrete Dirac operators with Stark potential}
\author{Moacir Aloisio, C\'esar de Oliveira and Mariane Pigossi}
\date{}

\begin{document}

\maketitle
\begin{abstract}
We study long-range discrete Dirac operators with Stark potential, extending the theory of Stark localization from scalar lattice models to systems with internal spinorial structure. We initially investigate the local setting, where two distinct localization mechanisms arise. The standard local Dirac-Stark operator yields two Stark-type spectral ladders and exponentially localized spinorial eigenfunctions. Conversely, a related pure-shift local model exhibits an invariant block structure that leads to explicitly computable eigenvalues and exact localization, with eigenfunctions compactly supported on only two spinorial sites. This extreme confinement surpasses the factorial decay characteristic of the classical scalar Stark model. For the general long-range Dirac model, we observe that the eigenvalues remain asymptotically close to the Stark ladder and prove that the corresponding eigenfunctions satisfy power-law localization estimates. Consequently, we establish power-law localization in the sense of finite moments of the position operator for the spinorial evolution. Our results demonstrate that deterministic Stark localization is robust and persists in genuinely matrix-valued lattice systems.
\end{abstract}
\sloppy

\renewcommand{\thetable}{\Alph{table}}

%%%%%%%%%%%%%%%%%%%%%%%%%%%%%%%%%%%%%%%%%%%%%%%%%%%%%%%%%%%%%%%%%%%%%%%%%%%%%%%%%%%%%%%%%%%%%%%%%%%%%%%%%%%%%%%%%%--Introduction--%%%%%%%%%%%%%%%%%%%%%%%%%%%%%%%%%%%%%%%%%%%%%%%%%%%%%%%%%%%%%%%%%%%%%%%%%%%%%%%%%%%%%%%%%%%%%%%%%%%%%%%%%%%%%%%%%%%%%%%%%%%%%%%%%%%%%%%%%%%%%%%%%%%%%%%%%%%%%%%%%%%%%%%%%%%%%%%%%%%%%%%

%SSSSS
\section{Introduction}

%sssss
\subsection{Contextualization}

Dynamical localization is one of the central manifestations of the absence of quantum transport in lattice systems. In contrast with purely spectral localization, which concerns the nature of the spectrum and the spatial decay of eigenfunctions, dynamical localization describes the confinement of initially localized wave packets uniformly in time. Since Anderson's seminal work on the absence of diffusion in random lattices \cite{And58}, localization has become a major topic in mathematical physics, leading to a variety of powerful techniques, including the fractional moment method \cite{Aizenman,Aizenman2}, multiscale analysis \cite{Spencer}, and approaches based on SULE and eigenfunction correlator estimates \cite{DJLS,Tcheremchantsev}.

A particularly interesting deterministic localization mechanism is provided by the Stark effect. In lattice systems subjected to a constant electric field, the linear Stark potential suppresses transport even in the absence of disorder. Earlier works on discrete Schr\"odinger operators established dynamical localization and related it to the Stark ladder structure of the spectrum and the decay of eigenfunctions \cite{Pigossi1,Pigossi,Nazareno}.

More recently, Stark localization has been extended to long-range hopping models with polynomially decaying interactions. Such operators arise naturally in several physical settings, including dipolar Frenkel excitons, nuclear spin systems, and the quantum Kepler model \cite{AltshulerL1997,Shi2}. Localization phenomena for long-range operators have been investigated in random, quasi-periodic, and Stark regimes \cite{Aizenman2,Aloisio,Aloisio2,Disertori,Sun2,Kerner,Kraisler,Shi2,Shi4,Shi5,Shi6,Sun}. In the Stark setting, power-law localization was first obtained through perturbative KAM-type arguments \cite{Sun}. Subsequently, a non-perturbative approach based on eigenvalue asymptotics, the Min-Max Principle, and power-law ULE estimates established power-law localization for arbitrary bounded perturbations \cite{Aloisio}. More broadly, Stark localization has also been investigated in related settings, including Jacobi operators, nonlinear Stark models, time quasi-periodic Hamiltonians, and interacting-particle systems \cite{Roeck,Hu2,Hu,Sun3}.

At the same time, discrete Dirac operators have attracted considerable attention as lattice models incorporating internal spinorial degrees of freedom. Besides their relevance to relativistic lattice systems and graphene-type models, they provide a natural framework for investigating localization phenomena in matrix-valued settings. Recent works have studied localization for discrete Dirac operators in random and quasi-random environments, including Bernoulli models, decaying random potentials, and quasi-one-dimensional systems \cite{Barbaroux,Boumaza,Bourget,BourgetVargasMancipe,Prado1,Prado2}. These studies indicate that localization mechanisms for Dirac operators may differ substantially from those arising in scalar Schr\"odinger models.

The purpose of the present work is to extend Stark localization theory to a class of long-range discrete Dirac operators. More precisely, we investigate operators of the form
\[
\mathcal L=D_a+X\otimes I_{\mathbb C^2}+B\otimes I_{\mathbb C^2},
\]
where \(D_a\) is a long-range discrete Dirac operator, \(X\otimes I_{\mathbb C^2}\) is the Stark potential acting equally on both spin components, and \(B\) is a bounded perturbation.  This choice corresponds to the massless case \(m=0\) of the one-dimensional discrete Dirac operator, in units where the characteristic velocity is normalized to \(c=1\) \cite{Prado1,Prado2}. The precise definition of \(D_a\) is given in Subsection~\ref{defoperators}.

The  question addressed here is whether the spectral-asymptotic mechanism underlying Stark localization survives in this matrix-valued setting. Our main result (see Theorem \ref{mainthm} ahead) shows that this is indeed the case, that is, the eigenvalues remain asymptotically organized along Stark ladders, the corresponding eigenfunctions exhibit power-law localization, and the associated spinorial dynamics satisfies power-law localization estimates.

To illustrate the distinctive features of the Dirac framework, we first analyze two local nearest-neighbor models. The standard local Dirac-Stark operator exhibits two Stark-type spectral ladders and exponentially localized eigenfunctions. More surprisingly, a related pure-shift Dirac-Stark model possesses a special off-diagonal structure that decomposes the infinite-dimensional problem into independent two-dimensional blocks. As a consequence, all eigenfunctions are compactly supported on exactly two spinorial sites. This exact localization is  stronger than the factorial decay characteristic of the classical scalar Stark model~\cite{Pigossi} and reveals a phenomenon that has no direct counterpart in the scalar setting.

Note that compactly supported eigenvectors occur for certain two-dimensional lattice systems, such as the quantum graph model for graphene and layers of graphene with Dirichlet boundary conditions~\cite{deORocha2020,KuchPost2007}, but this phenomenon is notably rare in one-dimensional lattice systems.

We summarize the main novelties of the paper as follows.

\begin{enumerate}

\item \textbf{A new model:  Dirac-Stark lattice model.}
We introduce a class of long-range discrete Dirac operators with Stark potential acting in the spinorial Hilbert space
\[
\ell^2(\mathbb Z)\otimes\mathbb C^2.
\]

\item \textbf{Exact localization in a local Dirac model.}
We identify a pure-shift local Dirac-Stark operator whose off-diagonal structure induces a decomposition into invariant two-dimensional blocks. This yields explicit eigenvalues and compactly supported eigenfunctions (Proposition~\ref{prop2}), providing an extreme form of localization that contrasts sharply with the factorial decay found in scalar Stark models.

\item \textbf{Robustness of Stark localization in the Dirac setting.}
We prove that the spectral-asymptotic mechanism responsible for localization in long-range Stark operators remains valid in the presence of spinorial degrees of freedom and spectral multiplicity. Our approach is entirely non-perturbative and does not rely on KAM techniques.

\end{enumerate}

The paper is organized as follows. In Subsection~1.2 we introduce the relevant operators, discuss the local Dirac-Stark models, and state the main theorem. Section~2 is devoted to the proofs of the local results. In Section~\ref{secmainthm} we establish the eigenvalue asymptotics, prove power-law localization estimates for the eigenfunctions, and derive the corresponding power-law localization bounds.

Throughout the paper, $\ell^2(\mathbb Z)$ denotes the Hilbert space of square-summable sequences on $\mathbb Z$, with canonical basis $\{\delta_n\}_{n\in\mathbb Z}$, and $\{e_1,e_2\}$ denotes the canonical basis of $\mathbb C^2$. We write
\[
\mathcal H := \ell^2(\mathbb Z)\otimes \mathbb C^2.
\]

For each $n\in\mathbb Z$, $P_n$ denotes the orthogonal projection of $\mathcal H$ onto
\[
\operatorname{span}\{\delta_n\otimes e_1, \delta_n\otimes e_2\}.
\]
Thus, if
$
\Psi = \psi_1\otimes e_1 + \psi_2\otimes e_2 \in \mathcal H,
$
then
\[
P_n\Psi = \psi_1(n)\delta_n\otimes e_1 + \psi_2(n)\delta_n\otimes e_2
\]
and
\[
\|P_n\Psi\|^2 = |\psi_1(n)|^2 + |\psi_2(n)|^2.
\]

Finally, throughout the paper, for every $n\in\mathbb Z$, we write
\[
\langle n\rangle := 1+|n|.
\]

%sssssss
\subsection{Long-range discrete Dirac operators with Stark potential} \label{defoperators}

Consider the Hilbert space \(\mathcal{H} = \ell^2(\mathbb Z)\otimes \mathbb C^2\). We denote by \(\{\delta_n\}_{n\in\mathbb Z}\) the canonical basis of \(\ell^2(\mathbb Z)\), and by \(\{e_1,e_2\}\) the canonical basis of \(\mathbb C^2\). Every vector \(\Psi\in \mathcal{H}\) can be written in the form
\[
\Psi = \psi_1\otimes e_1 + \psi_2\otimes e_2, \, \,  \psi_1,\psi_2\in \ell^2(\mathbb Z).
\]

Let \(a\in \ell^1_r(\mathbb Z)\),  \(r\geq 0\), that is, $\|a\|_r := \sum_{m\in\mathbb Z}|a(m)|\,|m|^r <\infty.$ Suppose that $a(0)=0$. We define the long-range hopping operator \(T_a\) on \(\ell^2(\mathbb Z)\) by
\[
(T_a u)(m) = \sum_{k\in\mathbb Z} a(k-m)u(k), \, \,  m\in\mathbb Z.
\]
Its adjoint is given by
\[
(T_a^*u)(m) = \sum_{k\in\mathbb Z} a^*(m-k)u(k), \, \,  m\in\mathbb Z.
\]
where \(a^*\) denotes the conjugate of \(a\).

The long-range discrete Dirac operator
\[
D_a:\mathcal{H} \longrightarrow \mathcal{H}
\]
is here introduced as
\[
D_a\Psi = (T_a^*-I)\psi_2\otimes e_1 + (T_a-I)\psi_1\otimes e_2.
\]
This operator may be viewed as a massless discrete Dirac operator whose kinetic part is obtained by replacing the nearest-neighbor difference operator with a long-range convolution operator.

Let \(X\) be the position operator on \(\ell^2(\mathbb Z)\), i.e.,
\[
(Xu)(n) = n u(n), \, \,  n\in\mathbb Z.
\]
The Stark potential is defined by
\[
V := X\otimes I_{\mathbb C^2}.
\]
Therefore,
\[
V\Psi = (X\psi_1)\otimes e_1 + (X\psi_2)\otimes e_2
\]
on the domain 
\[
\operatorname{dom} \, V = \Big\{ \Psi\in\mathcal H: \sum_{n\in\mathbb Z} n^2 |\psi_j(n)|^2 < \infty, \; j \in \{1,2\} \Big\}.
\]

In this work, we study the spectral and localization properties of the self-adjoint operator
\begin{equation}\label{eqmain0101}
\mathcal{L} = D_a + X\otimes I_{\mathbb{C}^2} + B \otimes I_{\mathbb{C}^2},    
\end{equation}
where $B \in \ell^\infty(\mathbb{Z},\mathbb{R})$ and $\operatorname{dom}\, \mathcal{L} = \operatorname{dom}\, V.$
In matrix representation, with respect to the standard spinorial basis $\{e_1, e_2\}$, the operator $\mathcal L$ acts on a vector $\Psi \equiv \begin{pmatrix} \psi_1 \\ \psi_2 \end{pmatrix}$ as 
\[
\mathcal{L} \begin{pmatrix} \psi_1 \\ \psi_2 \end{pmatrix} = \begin{pmatrix} X+B & T_a^*-I \\ T_a-I & X+B \end{pmatrix} \begin{pmatrix} \psi_1 \\ \psi_2 \end{pmatrix}.
\]

%sssssss
\subsection*{The local case}

We consider  the local case, which is obtained by taking
\[
a_{\mathrm{loc}}(m) = \delta_{m,1}, \, \,  m\in\mathbb Z.
\]
Under this assumption, the long-range discrete Dirac operator \(D_a\) reduces to the (classical) discrete Dirac operator \(D\). Namely, for every $\Psi \in \mathcal{H}$, we set
\[
D\Psi = (T_{a_{\mathrm{loc}}}^*-I)\psi_2\otimes e_1 + (T_{a_{\mathrm{loc}}}-I)\psi_1\otimes e_2.
\]

The discrete Dirac operator with Stark potential is $\mathcal{A}:=D+V$, with ${\rm dom} \, \mathcal{A}= {\rm dom} \, V$, which yeilds
\[
\mathcal{A}\Psi = \left(X\psi_1+(T_{a_{\mathrm{loc}}}^*-I)\psi_2\right)\otimes e_1 + \left((T_{a_{\mathrm{loc}}}-I)\psi_1+X\psi_2\right)\otimes e_2.
\]
In matrix representation, the operator $\mathcal A$ acts on a vector $\Psi \equiv \begin{pmatrix} \psi_1 \\ \psi_2 \end{pmatrix}$ as 
\[
\mathcal{A} \begin{pmatrix} \psi_1 \\ \psi_2 \end{pmatrix} = \begin{pmatrix} X & T_{a_{\mathrm{loc}}}^*-I \\ T_{a_{\mathrm{loc}}}-I & X \end{pmatrix} \begin{pmatrix} \psi_1 \\ \psi_2 \end{pmatrix}.
\]

\begin{proposition}\label{prop1}
The spectrum of \(\mathcal A\) is pure point and consists of the eigenvalues
\[
\lambda_{k,j} = k+j\frac{\beta}{2\pi}, \, \,  k\in\mathbb Z, \; j\in\{-1,1\},
\]
for some \(\beta\in[0,\pi]\). The corresponding eigenfunctions 
\[
\{\Psi_{m,s}\}_{m\in\mathbb Z,\; s\in\{-1,1\}}
\]
form an orthonormal basis of $\mathcal{H}$ and satisfy, for some $\alpha>0$ and  $\gamma>0$,
\[
\|P_n \Psi_{m,s}\|\leq \gamma \,  e^{-\alpha |m-n|}, \, \,  m,n\in\mathbb Z, \;  s\in\{-1,1\}.
\]
\end{proposition}

We now consider the pure-shift local case. Retaining $a_{\mathrm{loc}}(m)=\delta_{m,1}$, we define the pure-shift local discrete Dirac operator
\[
D_{\mathrm{sh}}\Psi = (T_{a_{\mathrm{loc}}}^*\psi_2)\otimes e_1 + (T_{a_{\mathrm{loc}}}\psi_1)\otimes e_2.
\]

The corresponding pure-shift local Dirac-Stark operator is $\mathcal S:=D_{\mathrm{sh}}+V$, yielding
\[
\mathcal S\Psi = \left(X\psi_1+T_{a_{\mathrm{loc}}}^*\psi_2\right)\otimes e_1 + \left(T_{a_{\mathrm{loc}}}\psi_1+X\psi_2\right)\otimes e_2.
\]
In matrix representation, the operator $\mathcal S$ acts on a vector $\Psi \equiv \begin{pmatrix} \psi_1 \\ \psi_2 \end{pmatrix}$ as 
\[
\mathcal{S} \begin{pmatrix} \psi_1 \\ \psi_2 \end{pmatrix} = \begin{pmatrix} X & T_{a_{\mathrm{loc}}}^* \\ T_{a_{\mathrm{loc}}} & X \end{pmatrix} \begin{pmatrix} \psi_1 \\ \psi_2 \end{pmatrix}.
\]

\begin{proposition}\label{prop2}
The spectrum of \(\mathcal S\) is pure point and consists of the eigenvalues
\[
\lambda_{k,j} = k+\frac12-j\frac{\sqrt5}{2}, \quad k\in\mathbb Z, \;  j\in\{-1,1\},
\]
with corresponding normalized eigenfunctions  (which forms an orthonormal basis of \(\mathcal{H}\))
\[
\Psi_{k,j} = \frac{\delta_{k+1}\otimes e_1 - \dfrac{1+j\sqrt5}{2}\, \delta_k\otimes e_2}{\sqrt{1+ \left(\dfrac{1+j\sqrt5}{2}\right)^2}}, \quad  k\in\mathbb Z, \; j\in\{-1,1\}.
\]
\end{proposition}

\begin{remark} {\rm
\begin{enumerate}
    \item [(i)] It follows from Propositions~\ref{prop1} and~\ref{prop2} that the operators \(\mathcal A\) and \(\mathcal S\) exhibit dynamical localization; that is,
for every initial condition \(\delta_k \otimes e_j\), with \(k \in \mathbb Z\) and \(j \in \{1,2\}\), all moments of order \(q>0\) of the position operator are uniformly bounded in time.

\item [(ii)] The operator $\mathcal S$ differs from the local Dirac--Stark operator $\mathcal A$ by the removal of the identity terms: its off-diagonal part contains only the shifts $S$ and $S^*$. This simplification reveals a special block structure, since the Dirac coupling pairs $\delta_{k+1}\otimes e_1$ only with $\delta_k\otimes e_2$. Hence, $\mathcal S$ decomposes into an orthogonal direct sum of two-dimensional invariant subspaces, reducing the spectral problem to a family of $2\times2$ matrices. This block decomposition also explains the exact localization of the eigenfunctions. In contrast with compact localization in discrete Schr\"odinger operators, which is usually caused by vanishing hopping terms and lattice disconnection, here all couplings remain nontrivial. The compact support instead follows from the precise alignment between the lattice shift and the internal spin degrees of freedom. 

    \item [(iii)]  The results of this paper are stated for the linear Stark potential. However, based on the recent theory developed in \cite{Aloisio2}, and on the methods introduced in the present work, it is expected that analogous results should extend naturally to the Dirac setting with sublinear Stark potentials.
\end{enumerate}
} \end{remark}

\begin{remark}{\rm
\emph{Physical interpretation of the exact localization.}
The exact localization observed for the pure-shift Dirac--Stark operator $\mathcal S$ may be understood as a consequence of a complete fragmentation of the transport channels available to the particle.

Indeed, the off-diagonal hopping couples only the pair of spinorial states
\[
\delta_{k+1}\otimes e_1
\qquad\text{and}\qquad
\delta_k\otimes e_2,
\]
while no coupling exists between different pairs. As a result, the infinite lattice decomposes into independent two-level systems (see the proof of Proposition~\ref{prop2} ahead). From the dynamical point of view, a particle initially placed inside one of these blocks can oscillate between the two corresponding spinorial sites, but it cannot propagate to neighboring blocks.

The Stark potential assigns different on-site energies to the two sites of each block, namely $k+1$ and $k$. Consequently, each invariant block becomes a two-level Stark system whose eigenvalues form a pair of energy levels centered around $k+\frac12$. As $k$ varies along the lattice, these independent two-level systems generate a sequence of Stark doublets, reproducing the ladder structure characteristic of Stark-type operators.

Since no mechanism couples different doublets, quantum transport across the lattice is  suppressed. The eigenstates are therefore confined to a single two-level subsystem and become exactly compactly supported.
} \end{remark}

%ssssss
\subsection{Main Result}

Now we present our main result.

\begin{theorem}\label{mainthm}
Let \(\mathcal L\) be the operator defined in \eqref{eqmain0101}. Then the following statements hold.

\begin{enumerate}
    \item[\emph{(i)}] There exists a constant \(\gamma>0\) such that the eigenvalues of \(\mathcal L\) can be enumerated, counting multiplicities, as \(\{\lambda_{m,s}\}_{m\in\mathbb Z,\ s=1,2}\), and satisfy
    \[
    |\lambda_{m,s}-m|\leq \gamma, \, \,  m\in\mathbb Z, \; s=1,2.
    \]

    \item[\emph{(ii)}] If \(a\in\ell^1_r(\mathbb Z)\), with \(r\in\mathbb N\cup\{0\}\), then any orthonormal basis of \(\mathcal{H}\) consisting of eigenfunctions of \(\mathcal L\), denoted \(\{\Phi_{m,s}\}\), satisfies
    \[
    \|P_n\Phi_{m,s}\| \leq \frac{\gamma_r}{\langle n-m\rangle^{r+1}}, \, \,  m,n\in\mathbb Z, \; s=1,2,
    \]
    where \(\gamma_r>0\) is a constant depending only on \(r\) and \(\|a\|_r\).

    \item[\emph{(iii)}] If $0<q<2r-1$, then, for every \(k\in\mathbb Z\) and every \(j \in \{1,2\}\), one has
    \[
    \sup_{t\in\mathbb R} \sum_{n\in\mathbb Z} |n|^q \left\| P_n e^{-it\mathcal L}(\delta_k\otimes e_j) \right\|^2 < \infty.
    \]
\end{enumerate}
\end{theorem}

We conclude the introduction by highlighting the main differences in the analytical treatment between the scalar model and the Dirac setting, which are summarized as follows:

\begin{enumerate}
    \item The first key step is to formulate the problem in the tensor-product space
    \[
    \mathcal{H}=\ell^2(\mathbb{Z})\otimes \mathbb{C}^2.
    \]
    Although this space is naturally isomorphic to $\ell^2(\mathbb{Z},\mathbb{C}^2)$, the tensorial formulation makes the structure of the model much clearer. It separates the lattice variable from the internal spin degree of freedom and allows localization to be expressed naturally through the projections
    \[
    P_n:\mathcal{H}\longrightarrow
    \operatorname{span}\{\delta_n\otimes e_1,\delta_n\otimes e_2\}.
    \]
    Thus, the relevant localized quantity is
    \[
    \|P_n\Phi\|^2
    =
    |\phi_1(n)|^2+|\phi_2(n)|^2,
    \]
    which treats both spin components simultaneously.

    \item A second technical difference is that the eigenvalue equation becomes a coupled system, i.e.,
    \[
    (n+B(n)-\lambda)\phi_1(n)
    +(T_a^*-I)\phi_2(n)=0,
    \]
    \[
    (n+B(n)-\lambda)\phi_2(n)
    +(T_a-I)\phi_1(n)=0.
    \]
    In the scalar problem, the decay estimate is obtained directly for a single eigenfunction component. In the Dirac case, however, neither component can be controlled independently, since each one is coupled to the other. Therefore, the inductive argument must be closed simultaneously for both components in terms of the vector norm $\|P_n\Phi\|$. 

\item 
In the Dirac setting, each lattice site carries a two-dimensional internal degree of freedom, leading to an intrinsic twofold multiplicity for the Stark operator $X \otimes I_{\mathbb{C}^2}$. Consequently, Stark levels naturally split into doublets, producing two spectral ladders in the local model. In contrast, the scalar Schr\"odinger operator has simple Stark levels with no internal degeneracy. This spinorial doubling is essential for the coupled structure of the eigenvalue problem and underlies the matrix-valued localization analysis and the block decomposition observed in the pure-shift case.
\end{enumerate}

%SSSSSS
\section{Proofs of the Propositions}

Next, we prove Proposition~\ref{prop1}. To this end, we first recall a standard estimate for Fourier coefficients of analytic periodic functions.

\begin{lemma}[Theorem 4.1 in \cite{Wright}] 
\label{Fourierlemma}
Let $f$ be a $2\pi$-periodic analytic function on the strip
\[
S_\alpha
=
\{z\in\mathbb C:\ |\operatorname{Im}z|<\alpha\},
\]
for some $\alpha>0$, and suppose that
\[
|f(z)|\le \gamma,
\qquad z\in S_\alpha .
\]

Define
\[
c_n=
\frac1{2\pi}
\int_0^{2\pi}
f(\theta)e^{-in\theta}\,d\theta,
\qquad n\in\mathbb Z.
\]

Then
\[
|c_n|
\le
\gamma e^{-\alpha |n|},
\qquad n\in\mathbb Z.
\]
\end{lemma}

\subsubsection*{Proof of Proposition \ref{prop1}}

Let
\[
\Psi
=
\psi_1\otimes e_1+\psi_2\otimes e_2
\in {\rm dom}\,\mathcal A.
\]
The eigenvalue equation
\[
\mathcal A\Psi=\lambda\Psi
\]
is equivalent to
\[
\begin{cases}
X\psi_1+(T_{a_{\rm loc}}^*-I)\psi_2
=
\lambda\psi_1,
\\[0.3em]
(T_{a_{\rm loc}}-I)\psi_1+X\psi_2
=
\lambda\psi_2.
\end{cases}
\]

Applying the Fourier transform
\begin{equation}
\label{Fourier transf}
\mathcal F:\ell^2(\mathbb Z)\to L^2(\mathbb T),
\qquad
\widehat u(\theta)
=
\sum_{n\in\mathbb Z}
u(n)e^{-in\theta},
\end{equation}
and using
\[
\widehat{Xu}
=
i\partial_\theta \widehat u,
\qquad
\widehat{T_{a_{\rm loc}}u}
=
e^{i\theta}\widehat u,
\qquad
\widehat{T_{a_{\rm loc}}^*u}
=
e^{-i\theta}\widehat u,
\]
we obtain
\[
\begin{cases}
i\partial_\theta\widehat\psi_1
+
(e^{-i\theta}-1)\widehat\psi_2
=
\lambda\widehat\psi_1,
\\[0.3em]
(e^{i\theta}-1)\widehat\psi_1
+
i\partial_\theta\widehat\psi_2
=
\lambda\widehat\psi_2.
\end{cases}
\]

Writing
\[
\widehat\Psi(\theta)
=
\begin{pmatrix}
\widehat\psi_1(\theta)
\\
\widehat\psi_2(\theta)
\end{pmatrix},
\]
this system becomes
\[
i\partial_\theta\widehat\Psi(\theta)
+
P(\theta)\widehat\Psi(\theta)
=
\lambda\widehat\Psi(\theta),
\]
where
\[
P(\theta)
=
\begin{pmatrix}
0 & e^{-i\theta}-1
\\
e^{i\theta}-1 & 0
\end{pmatrix}.
\]

Equivalently,
\[
\partial_\theta\widehat\Psi(\theta)
=
-i\lambda\widehat\Psi(\theta)
+
iP(\theta)\widehat\Psi(\theta).
\]

Let $Y(\theta)$ denote the fundamental matrix solution of
\[
Y'(\theta)
=
iP(\theta)Y(\theta),
\qquad
Y(0)=I.
\]

Since $P(\theta)$ is an entire matrix-valued function of $\theta$, standard ODE theory implies that $Y(\theta)$ is entire in $\theta$.

Every solution of the eigenvalue equation is therefore of the form
\[
\widehat\Psi(\theta)
=
e^{-i\lambda\theta}
Y(\theta)v,
\]
for some $v\in\mathbb C^2$. Since $\widehat\Psi$ must be $2\pi$-periodic, we impose
\[
\widehat\Psi(2\pi)
=
\widehat\Psi(0).
\]
Hence $
e^{-i2\pi\lambda}Y(2\pi)v=v.$

Set $
M:=Y(2\pi).$ 
The periodicity condition becomes
\begin{equation}
\label{eq010101}
e^{-i2\pi\lambda}Mv=v.
\end{equation}

Because $P(\theta)$ is self-adjoint,
\[
\frac{d}{d\theta}
\bigl(
Y(\theta)^*Y(\theta)
\bigr)
=
0,
\]
and therefore $Y(\theta)$ is unitary for every $\theta$. Furthermore,
$
\operatorname{tr}P(\theta)=0.$
By Liouville's formula,
\[
\det Y(\theta)
=
\exp
\left(
\int_0^\theta
i\,\operatorname{tr}P(s)\,ds
\right)
=
1.
\]
Consequently,
\[
M\in SU(2),
\]
and  the eigenvalues of $M$ are $
e^{i\beta},$  $e^{-i\beta},$
for some $\beta\in[0,\pi]$.

Let $j\in\{-1,1\}$ and let $v_j$ be a normalized eigenvector satisfying
\[
Mv_j=e^{ij\beta}v_j.
\]
Substituting into \eqref{eq010101} gives
\[
e^{-i2\pi\lambda}e^{ij\beta}=1.
\]

Hence
\[
\lambda
=
k+j\frac{\beta}{2\pi},
\qquad
k\in\mathbb Z.
\]

Thus the eigenvalues are
\[
\lambda_{k,j}
=
k+j\frac{\beta}{2\pi},
\qquad
k\in\mathbb Z,
\quad
j\in\{-1,1\}.
\]

The corresponding normalized eigenfunctions in Fourier representation are
\[
\widehat\Psi_{k,j}(\theta)
=
\frac1{\sqrt{2\pi}}
e^{-i\lambda_{k,j}\theta}
Y(\theta)v_j.
\]

To prove exponential localization, define
\[
G_j(\theta)
=
\frac1{\sqrt{2\pi}}
e^{-ij\beta\theta/(2\pi)}
Y(\theta)v_j.
\]

Then
\[
\widehat\Psi_{m,j}(\theta)
=
e^{-im\theta}G_j(\theta).
\]

Taking Fourier coefficients yields
\[
\Psi^{(r)}_{m,j}(n)
=
\Psi^{(r)}_{0,j}(n-m),
\qquad
r=1,2.
\]

Since $Y(\theta)$ is entire, every component of $G_j$ is entire as well. In particular, for every $\alpha>0$ there exists a constant $\gamma_\alpha$ such that
\[
|G_j(\theta)|
\le
\gamma_\alpha,
\qquad
|\operatorname{Im}\theta|<\alpha.
\]

Applying Lemma~\ref{Fourierlemma} componentwise, we obtain a constant
$\gamma'_\alpha$ such that
\[
|\Psi_{0,j}^{(r)}(n)|
\le
\gamma'_\alpha e^{-\alpha|n|},
\qquad
n\in\mathbb Z,
\quad
r=1,2.
\]

Therefore,
\[
\|P_n\Psi_{m,j}\|
=
\|\Psi_{m,j}(n)\|_{\mathbb C^2}
=
\|\Psi_{0,j}(n-m)\|_{\mathbb C^2}
\le
\sqrt2\,\gamma'_\alpha
e^{-\alpha|m-n|}.
\]

Thus
\[
\|P_n\Psi_{m,j}\|
\le
\sqrt2\,\gamma'_\alpha
e^{-\alpha|m-n|}.
\]

Since the Floquet condition above yields exactly two independent eigenvectors for each $k\in\mathbb Z$, the family
\[
\{\Psi_{m,j}\}_{m\in\mathbb Z,\ j\in\{-1,1\}}
\]
is complete in $\mathcal H$. Choosing the vectors $v_j$ orthonormally, this family becomes an orthonormal basis of $\mathcal H$. The proof is complete.
\hfill $\square$

\subsubsection*{Proof of Proposition \ref{prop2}}

We exploit the  structure of the pure-shift operator
\[
\mathcal{S} =
\begin{pmatrix}
X & T_{a_{\mathrm{loc}}}^* \\
T_{a_{\mathrm{loc}}} & X
\end{pmatrix},
\qquad a_{\mathrm{loc}}(m)=\delta_{m,1},
\]
with
\[
(T_{a_{\mathrm{loc}}}u)(n)=u(n-1), 
\qquad 
(T_{a_{\mathrm{loc}}}^*u)(n)=u(n+1).
\]
The operator $\mathcal{S}$ acts  as
\[
(\mathcal{S}\Psi)(n)
=
\begin{pmatrix}
n & 1 \\
1 & n
\end{pmatrix}
\Psi(n),
\]
and it preserves the two-dimensional subspaces
\[
\mathcal{H}_k = \operatorname{span}\{\delta_{k+1}\otimes e_1,\; \delta_k\otimes e_2\},
\qquad k\in\mathbb{Z}.
\]

Indeed, a direct computation shows that
\[
\mathcal{S}(\delta_{k+1}\otimes e_1)
=
(k+1)\delta_{k+1}\otimes e_1 + \delta_k\otimes e_2,
\]
\[
\mathcal{S}(\delta_k\otimes e_2)
=
\delta_{k+1}\otimes e_1 + k\,\delta_k\otimes e_2.
\]

Hence each $\mathcal{H}_k$ is invariant under $\mathcal{S}$, and the restriction
$\mathcal{S}|_{\mathcal{H}_k}$ is represented, in the basis
\(\{\delta_{k+1}\otimes e_1,\delta_k\otimes e_2\}\), by the matrix
\[
S_k =
\begin{pmatrix}
k+1 & 1 \\
1 & k
\end{pmatrix},
\]
and
\[
\mathcal{S} = \bigoplus_{k\in\mathbb{Z}} S_k.
\]

The characteristic polynomial of $S_k$ is
\[
\det(S_k-\lambda I)
=
(k+1-\lambda)(k-\lambda)-1.
\]
and to obtain its eigenvalues
\[
\lambda^2 - (2k+1)\lambda + (k(k+1)-1)=0,
\]
which are given by
\[
\lambda_{k,\pm}
=
k+\frac{1}{2}\pm \frac{\sqrt{5}}{2}.
\]

For each eigenvalue $\lambda_{k,j}$, $j=\pm 1$, an eigenvector is given by
\[
v_{k,j}
=
\begin{pmatrix}
1 \\
-\dfrac{1+j\sqrt{5}}{2}
\end{pmatrix}.
\]
Thus the corresponding eigenfunction in $\mathcal{H}_k$ is
\[
\Psi_{k,j}
=
\delta_{k+1}\otimes e_1
-
\frac{1+j\sqrt{5}}{2}\,\delta_k\otimes e_2.
\]
We keep the same notation $\Psi_{k,j}$ for its normalized version.

The subspaces $\mathcal{H}_k$ are mutually orthogonal and their direct sum
satisfies
\[
\mathcal{H} = \bigoplus_{k\in\mathbb{Z}} \mathcal{H}_k.
\]
Since each $S_k$ is diagonalizable with an orthonormal eigenbasis,
the family $\{\Psi_{k,j}\}_{k\in\mathbb{Z},\,j=\pm 1}$ forms an orthonormal basis of $\mathcal{H}$.

It then follows that $\mathcal{S}$ has pure point spectrum
\[
\sigma(\mathcal{S})=
\left\{
k+\frac{1}{2}\pm\frac{\sqrt{5}}{2}
:\; k\in\mathbb{Z}
\right\},
\]
and eigenfunctions with compact support on exactly two lattice sites, completing the proof.
\hfill $\square$

%SSSSSS
\section{Proof of Theorem \ref{mainthm}}\label{secmainthm}

Next, we prove Theorem \ref{mainthm}. To this end, some preliminary results are required. 

\begin{theorem}\label{thmeigvalues} Let \(\ell^2(\mathbb Z)\otimes\mathbb C^2\) and let \(H=A+X\otimes I_{\mathbb C^2}\), where \(A\) is a bounded self-adjoint
operator on \(\ell^2(\mathbb Z)\otimes\mathbb C^2\). Then \(H\) is self-adjoint and has purely discrete spectrum. Moreover, its eigenvalues can be enumerated, counting multiplicities, as \(\{\lambda_{n,s}\}_{n\in\mathbb Z,\ s \in \{1,2\}}\), and
\[
|\lambda_{n,s}-n|\leq \|A\|,\, \, n\in\mathbb Z,\ s \in \{1,2\}.
\]
\end{theorem}

\begin{remark}\normalfont
Using the second resolvent identity, one verifies that \(H\) has compact resolvent. Hence, its spectrum is discrete. In this case, the proof of Theorem \ref{thmeigvalues} follows from the Min--Max Principle as in \cite{Aloisio}. We note that the only difference is that the multiplication operator $X\otimes I_{\mathbb C^2}$ has eigenvalues \(n\in\mathbb Z\), each with multiplicity two. Namely,
\[
(X\otimes I_{\mathbb C^2})(\delta_n\otimes e_s)
=
n(\delta_n\otimes e_s),
\, \, n\in\mathbb Z,\ s\in\{1,2\}.
\]
Therefore, the same Min--Max argument gives two eigenvalues near each
Stark level \(n\). Hence the eigenvalues of \( H\), counted with
multiplicities, can be enumerated as
\[
\{\lambda_{n,s}\}_{n\in\mathbb Z,\ s \in \{1,2\}},
\]
and satisfy 
\[
|\lambda_{n,s}-n|\leq \|A\|,
\, \, n\in\mathbb Z,\ s \in \{1,2\}.
\]
\end{remark}

\begin{lemma}[Power-law ULE]\label{mainlemma} Let \(H\) be a self-adjoint operator on $\ell^2(\mathbb Z)\otimes \mathbb C^2$.
Assume that \(H\) possesses an orthonormal basis of eigenfunctions $\{\Phi_{m,s}\}_{m\in\mathbb Z,\ s\in\{1,2\}}$ associated with eigenvalues $\lambda_{m,s}$, such that
\[
H\Phi_{m,s}=\lambda_{m,s}\Phi_{m,s},
\]
and suppose that, for some \(\alpha>0\),
\[
\|P_n\Phi_{m,s}\|
\leq
\frac{\gamma_\alpha}{\langle n-m\rangle^\alpha},
\, \,
m,n\in\mathbb Z,\ s\in\{1,2\}.
\]
If
\[
\alpha>\frac32+\frac q2,
\]
then, for every \(k\in\mathbb Z\) and every \(j\in\{1,2\}\),
\[
\sup_{t\in\mathbb R}
\sum_{n\in\mathbb Z}
|n|^q
\bigl\|P_n e^{-it H}(\delta_k\otimes e_j)\bigr\|^2
<\infty .
\]
\end{lemma}

\begin{proof}
Fix \(k\in\mathbb Z\) and \(j\in\{1,2\}\). Since \(\{\Phi_{m,s}\}_{m\in\mathbb Z,\ s\in\{1,2\}}\) is an orthonormal basis of eigenfunctions of \(H\), we have 
\[
e^{-it H}(\delta_k\otimes e_j)
=
\sum_{m\in\mathbb Z}\sum_{s=1}^2
e^{-it\lambda_{m,s}}
\langle \delta_k\otimes e_j,\Phi_{m,s}\rangle
\Phi_{m,s}.
\]
Applying \(P_n\), we get
\[
P_n e^{-it H}(\delta_k\otimes e_j)
=
\sum_{m\in\mathbb Z}\sum_{s=1}^2
e^{-it\lambda_{m,s}}
\langle \delta_k\otimes e_j,\Phi_{m,s}\rangle
P_n\Phi_{m,s}.
\]
Therefore, by the triangle inequality,
\[
\bigl\|P_n e^{-it H}(\delta_k\otimes e_j)\bigr\|
\leq
\sum_{m\in\mathbb Z}\sum_{s=1}^2
\bigl|\langle \delta_k\otimes e_j,\Phi_{m,s}\rangle\bigr|
\,\|P_n\Phi_{m,s}\|.
\]
Since \(\delta_k\otimes e_j\in \operatorname{ran} P_k\), it follows that
\[
\bigl|\langle \delta_k\otimes e_j,\Phi_{m,s}\rangle\bigr|
=
\bigl|\langle \delta_k\otimes e_j,P_k\Phi_{m,s}\rangle\bigr|
\leq
\|P_k\Phi_{m,s}\|.
\]
Thus,
\[
\bigl\|P_n e^{-it H}(\delta_k\otimes e_j)\bigr\|
\leq
\sum_{m\in\mathbb Z}\sum_{s=1}^2
\|P_k\Phi_{m,s}\|\,\|P_n\Phi_{m,s}\|.
\]
Using the assumed power-law decay, we obtain
\[
\bigl\|P_n e^{-it H}(\delta_k\otimes e_j)\bigr\|
\leq
2\gamma_\alpha^2
\sum_{m\in\mathbb Z}
\frac{1}{\langle k-m\rangle^\alpha \langle n-m\rangle^\alpha}.
\]

We now estimate the convolution term. Since finitely many indices do not
affect the boundedness of the moment, we may restrict the analysis to
\(|n-k|\geq 2\) and \(|n|\geq 2\). Choose \(\varepsilon_0>0\) such that
\[
\alpha>\frac32+\frac q2+\frac{\varepsilon_0}{2}.
\]
Then, as in the scalar case (see \cite{Aloisio}), the product is summable in \(m\), and there exists a constant \(C_{\alpha,k}>0\) such that
\[
\sum_{m\in\mathbb Z}
\frac{1}{\langle k-m\rangle^\alpha \langle n-m\rangle^\alpha}
\leq
\frac{C_{\alpha,k}}{\langle n-k\rangle^{\alpha-1-\varepsilon_0/2}}.
\]
Thus,
\[
\bigl\|P_n e^{-it H}(\delta_k\otimes e_j)\bigr\|
\leq
\frac{C_{\alpha,k}}{\langle n-k\rangle^{\alpha-1-\varepsilon_0/2}},
\]
uniformly in \(t\in\mathbb R\). Hence,
\[
\sum_{n\in\mathbb Z}
|n|^q
\bigl\|P_n e^{-it H}(\delta_k\otimes e_j)\bigr\|^2
\leq
C_{\alpha,k}
\sum_{n\in\mathbb Z}
\frac{|n|^q}{\langle n-k\rangle^{2\alpha-2-\varepsilon_0}}.
\]
By the choice of \(\varepsilon_0\),
\[
2\alpha-2-\varepsilon_0>q+1,
\]
and therefore the last series converges. Thus, there exists a constant
\(C_{q,k}>0\)  such that
\[
 \sup_{t\in\mathbb R} \sum_{n\in\mathbb Z}
|n|^q
\bigl\|P_n e^{-itH}(\delta_k\otimes e_j)\bigr\|^2
< C_{q,k}.
\]
\end{proof}

%ssssss
\subsection*{Proof of Theorem \ref{mainthm}}

Note that item (i) follows directly from Theorem 2.2 by taking $A = D_a + B \otimes I_{\mathbb{C}^2}.$ Moreover, by Lemma \ref{mainlemma}, item (ii) implies item (iii). Therefore, it remains only to prove item (ii).

Let
\[
\Phi_{m,s}
=
\phi_{m,s}^{(1)}\otimes e_1
+
\phi_{m,s}^{(2)}\otimes e_2
\in\mathcal H
\]
be a normalized eigenfunction of \(\mathcal L\), that is,
\[
\mathcal L\Phi_{m,s}
=
\lambda_{m,s}\Phi_{m,s}.
\]

For each \(n\in\mathbb Z\), we identify
\[
P_n\Phi_{m,s}
=
\begin{pmatrix}
\phi_{m,s}^{(1)}(n)\\[1mm]
\phi_{m,s}^{(2)}(n)
\end{pmatrix}
\in\mathbb C^2.
\]

Recall that
\[
D_a
=
\begin{pmatrix}
0&T_a^*-I\\
T_a-I&0
\end{pmatrix}.
\]
By the definition of \(T_a\),
\[
(T_a f)(n)
=
\sum_{\ell\in\mathbb Z}a(\ell-n)f(\ell)
=
\sum_{k\in\mathbb Z}a(-k)f(n-k),
\]
whereas
\[
(T_a^*f)(n)
=
\sum_{k\in\mathbb Z}a^*(k)f(n-k).
\]

For each \(k\in\mathbb Z\), set
\[
K(k)
:=
\begin{pmatrix}
0&a^*(k)\\[1mm]
a(-k)&0
\end{pmatrix}.
\]
Then
\[
\sum_{k\in\mathbb Z}K(k)P_{n-k}\Phi_{m,s}
=
\begin{pmatrix}
(T_a^*\phi_{m,s}^{(2)})(n)\\[1mm]
(T_a\phi_{m,s}^{(1)})(n)
\end{pmatrix}.
\]

Since \(a(0)=0\), we have $K(0)=0.$ Thus,
\[
\sum_{k\in\mathbb Z}K(k)P_{n-k}\Phi_{m,s}
=
\sum_{k\neq0}K(k)P_{n-k}\Phi_{m,s}.
\]
Evaluating the eigenvalue equation
\[
\mathcal L\Phi_{m,s}
=
\lambda_{m,s}\Phi_{m,s}
\]
at the site \(n\), we obtain
\[
\sum_{k\neq0}K(k)P_{n-k}\Phi_{m,s}
-
\begin{pmatrix}
0&1\\
1&0
\end{pmatrix}
P_n\Phi_{m,s}
+
\bigl(n+B(n)\bigr)P_n\Phi_{m,s}
=
\lambda_{m,s}P_n\Phi_{m,s}.
\]
Equivalently,
\[
\left[
\bigl(\lambda_{m,s}-n-B(n)\bigr)I_{\mathbb C^2}
+
\begin{pmatrix}
0&1\\
1&0
\end{pmatrix}
\right]
P_n\Phi_{m,s}
=
\sum_{k\neq0}K(k)P_{n-k}\Phi_{m,s}.
\]
Taking norms in \(\mathbb C^2\) and using the triangle inequality, we obtain
\[
\begin{aligned}
\bigl|\lambda_{m,s}-n-B(n)\bigr|
\|P_n\Phi_{m,s}\|
&\leq
\sum_{k\neq0}
\|K(k)P_{n-k}\Phi_{m,s}\|
\\
&\, \, +
\left\|
P_n\Phi_{m,s}
\right\|.
\end{aligned}
\]

By item~\textup{(i)}, there exists \( \gamma>0\) such that
\[
|\lambda_{m,s}-m-B(n)|
\leq\gamma, \, \, m,n \in \mathbb{Z}, \, s \in \{1,2\}.
\]
Note that
\[
\lambda_{m,s}-n-B(n)
=
(m-n)+\lambda_{m,s}-m-B(n).
\]
Therefore, by the reverse triangle inequality,
\[
\begin{aligned}
|\lambda_{m,s}-n-B(n)|
&\geq
|m-n|
-
|\lambda_{m,s}-m-B(n)|
\\
&\geq
|m-n|-\gamma.
\end{aligned}
\]
If
\[
|m-n|>2\gamma,
\]
then
\[
|m-n|-\gamma
\geq
\frac{|m-n|}{2}.
\]
Hence,
\[
|\lambda_{m,s}-n-B(n)|
\geq
\frac{|m-n|}{2}.
\]
Combining the inequalities above, we obtain
\[
\frac{|m-n|}{2}
\|P_n\Phi_{m,s}\|
\leq
\sum_{k\neq0}
\|K(k)\|
\|P_{n-k}\Phi_{m,s}\|
+
\|P_n\Phi_{m,s}\|.
\]
Thus,
\[
\|P_n\Phi_{m,s}\|
\leq
\frac{2}{|m-n|}
\sum_{k\neq0}
\|K(k)\|
\|P_{n-k}\Phi_{m,s}\|
+
\frac{2}{|m-n|}
\|P_n\Phi_{m,s}\|.
\]

Now choose \(R>0\) such that $R>2\gamma \, \, \text{and}\, \,  R\geq4.$ For \(|m-n|>R\), one has
\[
\frac{2}{|m-n|}
\leq\frac12.
\]
Hence, 
\[
\left(
1-\frac{2}{|m-n|}
\right)
\|P_n\Phi_{m,s}\|
\leq
\frac{2}{|m-n|}
\sum_{k\neq0}
\|K(k)\|
\|P_{n-k}\Phi_{m,s}\|.
\]
Since
\[
1-\frac{2}{|m-n|}
\geq\frac12,
\]
we conclude that
\begin{equation}\label{eqmainrec}
\|P_n\Phi_{m,s}\|
\leq
\frac{4}{|m-n|}
\sum_{k\neq0}
\|K(k)\|
\|P_{n-k}\Phi_{m,s}\|,
\, \, 
|m-n|>R.    
\end{equation}

We proceed by induction on \(r\in\mathbb{N}\cup\{0\}\).

First, let \(r=0\). Since \(\Phi_{m,s}\) is normalized, we have
\[
\|P_{n-k}\Phi_{m,s}\|\leq 1
\]
for every \(k\in\mathbb{Z}\). Therefore, for \(|m-n|>R\), the estimate in \eqref{eqmainrec}  gives
\[
\|P_n\Phi_{m,s}\|
\leq
\frac{4}{|m-n|}
\sum_{k\neq 0}\|K(k)\|.
\]

Since
\[
K(k)=
\begin{pmatrix}
0 & a^*(k)\\
a(-k) & 0
\end{pmatrix},
\]
we have
\[
\|K(k)\|
=
\max\{|a(k)|,|a(-k)|\}
\leq
|a(k)|+|a(-k)|.
\]
Hence,
\[
\sum_{k\neq 0}\|K(k)\|
\leq
2\sum_{k\in\mathbb{Z}}|a(k)|
=
2\|a\|_0.
\]
Thus,
\[
\|P_n\Phi_{m,s}\|
\leq
\frac{8\|a\|_0}{|m-n|},
\, \,  |m-n|>R.
\]
For \(|m-n|\leq R\), we use the trivial estimate $\|P_n\Phi_{m,s}\|\leq 1.$ Therefore, there exists a  \(\gamma_0>0\) such that
\[
\|P_n\Phi_{m,s}\|
\leq
\frac{\gamma_0}{\langle m-n\rangle},
\, \, 
m,n\in\mathbb{Z},\, \,  s=1,2.
\]
This proves the case \(r=0\).

Now suppose that the result holds for some
\(r\in\mathbb{N}\cup\{0\}\). Namely, assume that, whenever
\(a\in\ell^1_r(\mathbb{Z})\), one has
\[
\|P_n\Phi_{m,s}\|
\leq
\frac{\gamma_r}{\langle m-n\rangle^{r+1}},
\, \, 
m,n\in\mathbb{Z},\, \,  s=1,2.
\]

We prove the corresponding estimate for \(r+1\). Assume that $a\in\ell^1_{r+1}(\mathbb{Z}).$ Since $\ell^1_{r+1}(\mathbb{Z})\subset\ell^1_r(\mathbb{Z}),$ the induction hypothesis applies.

Let \(|m-n|>R\). From the  estimate in \eqref{eqmainrec}, we have
\[
\|P_n\Phi_{m,s}\|
\leq
\frac{4}{|m-n|}
\sum_{k\neq 0}
\|K(k)\|\,\|P_{n-k}\Phi_{m,s}\|.
\]
Since $K(0) =0$, we split the sum into two parts
\begin{eqnarray}\label{eqmaineq0202}\nonumber
\|P_n\Phi_{m,s}\|
&\leq&
\frac{4}{|m-n|}
\sum_{0<|k|\leq |m-n|/2}
\|K(k)\|\,\|P_{n-k}\Phi_{m,s}\|
\\
&+&
\frac{4}{|m-n|}
\sum_{|k|>|m-n|/2}
\|K(k)\|\,\|P_{n-k}\Phi_{m,s}\|.
\end{eqnarray}

For \(0<|k|\leq |m-n|/2\), the reverse triangle inequality gives
\[
|m-(n-k)|
=
|m-n+k|
\geq
|m-n|-|k|
\geq
\frac{|m-n|}{2}.
\]
Therefore, by the induction hypothesis,
\[
\|P_{n-k}\Phi_{m,s}\|
\leq
\frac{\gamma_r}
{\langle m-n+k\rangle^{r+1}}
\leq
\frac{2^{r+1}\gamma_r}
{|m-n|^{r+1}}.
\]
Hence,
\[
\begin{aligned}
&\frac{4}{|m-n|}
\sum_{0<|k|\leq |m-n|/2}
\|K(k)\|\,\|P_{n-k}\Phi_{m,s}\|
\\
&\, \, \leq
\frac{2^{r+3}\gamma_r}
{|m-n|^{r+2}}
\sum_{k\neq 0}\|K(k)\|.
\end{aligned}
\]

For the second part of the sum, we use
\[
\|P_{n-k}\Phi_{m,s}\|\leq 1.
\]
Namely, one has
\[
\begin{aligned}
&\frac{4}{|m-n|}
\sum_{|k|>|m-n|/2}
\|K(k)\|\,\|P_{n-k}\Phi_{m,s}\|
\\
&\, \, \leq
\frac{4}{|m-n|}
\sum_{|k|>|m-n|/2}\|K(k)\|.
\end{aligned}
\]

Since \(|k|>|m-n|/2\), we have
\[
1
\leq
\frac{2^{r+1}|k|^{r+1}}
{|m-n|^{r+1}}.
\]
Therefore,
\[
\sum_{|k|>|m-n|/2}\|K(k)\|
\leq
\frac{2^{r+1}}
{|m-n|^{r+1}}
\sum_{k\neq 0}|k|^{r+1}\|K(k)\|.
\]
Substituting this estimate into \eqref{eqmaineq0202}, we obtain
\[
\begin{aligned}
&\frac{4}{|m-n|}
\sum_{|k|>|m-n|/2}
\|K(k)\|\,\|P_{n-k}\Phi_{m,s}\|
\\
&\, \, \leq
\frac{2^{r+3}}
{|m-n|^{r+2}}
\sum_{k\neq 0}|k|^{r+1}\|K(k)\|.
\end{aligned}
\]
Moreover,
\[
\|K(k)\|
\leq
|a(k)|+|a(-k)|,
\]
and therefore
\[
\sum_{k\neq 0}|k|^{r+1}\|K(k)\|
\leq
2\|a\|_{r+1}.
\]
Also,
\[
\sum_{k\neq 0}\|K(k)\|
\leq
2\|a\|_0
\leq
2\|a\|_{r+1}.
\]

Combining the inequalities obtained above, we conclude that there exists
a constant \(\gamma_{r+1}>0\) such that
\[
\|P_n\Phi_{m,s}\|
\leq
\frac{\gamma_{r+1}}
{|m-n|^{r+2}},
\, \,  |m-n|>R.
\]

For \(|m-n|\leq R\), the estimate $\|P_n\Phi_{m,s}\|\leq 1$ holds. Thus, after increasing the constant if necessary, there exists
\(\gamma_{r+1}>0\) such that
\[
\|P_n\Phi_{m,s}\|
\leq
\frac{\gamma_{r+1}}
{\langle m-n\rangle^{r+2}},
\, \, 
m,n\in\mathbb{Z},\, \,  s=1,2.
\]
This completes the induction. Therefore, for every
\(r\in\mathbb{N}\cup\{0\}\), if \(a\in\ell^1_r(\mathbb{Z})\), there exists a $\gamma_r>0$ such that
\[
\|P_n\Phi_{m,s}\|
\leq
\frac{\gamma_r}
{\langle n-m\rangle^{r+1}},
\, \, 
m,n\in\mathbb{Z},\, \,  s=1,2.
\]
This proves item \textup{(ii)}.
\hfill \qedsymbol

%%%%%%%%%%%%%%%%%%%%%%%%%%%%%%%%%%%%%%%%%%%%%%%%%%%%%%%%%%%%%%%%%%%%%%%%%%%%%%%%%%%%%%%%%%%%%%%%%%%%%%%%%%%%%%%%%%%%%%%%%%%%%%%%%%%%%%%%%%%%%%%%%%%%%%%%%%%%%%%%%%%%%%%%%%%%%%%%%%%%%%%%%%%%%%%%%%%%%%%%%%%%%%%%%%%%%%%%%%%%%%%%%%%%%%%%%%%%%%%%%%%%%%%%%%%%%%%%%%%%%%%%%%%%%%%%%%%%%%%%%%%%%%%%%%%%%%%%%%

\newpage

\begin{center} \Large{Acknowledgments} 
\end{center}
\addcontentsline{toc}{section}{Acknowledgments}

\noindent M. Aloisio was supported by grant \#2025/25338-1 from the São Paulo Research Foundation (FAPESP) and in part by grant \#01/24/APQ-03132-24 from the Minas Gerais Research Foundation (FAPEMIG). C. R. de Oliveira thanks CNPq (a Brazilian government agency) for partial support under grant 303689/2021-8.

%%%%%%%%%%%%%%%%%%%%%%%%%%%%%%%%%%%%%%%%%%%%%%%%%%%%%%%%%%%%%%%%%%%%%%%%%%%%%%%%%%%%%%%%%%%%%%%%%%%%%%%%%%%%%%%%%%%%%%%%%%%%%%%%%%%%%%%%%%%%%%%%%thebibliography- %%%%%%%%%%%%%%%%%%%%%%%%%%%%%%%%%%%%%%%%%%%%%%%%%%%%%%%%%%%%%%%%%%%%%%%%%%%%%%%%%%%%%%%%%%%%%%%%%%%%%%%%%%%%%%%%%%%%%%%%%%%%%%%%%%%%%%%%%%%%%%%%%%%%%%%%%%%%%%%%%%%%%%%%%%%%%%%%%

\ 

\noindent  \noindent Moacir Aloisio. Email: moacir.aloisio@ufvjm.edu.br, DME, UFVJM, Diamantina, MG, 39100-000, Brazil and ICMC, USP, S\~ao Carlos, SP, 13566-590, Brazil

\noindent  C\'esar R. de Oliveira. Email: oliveira@ufscar.br,  DM,   UFSCar, S\~ao Carlos, SP, 13560-970 Brazil

\noindent  Mariane Pigossi. Email: mariane.pigossi@ufes.br, DMAT, UFES, Vit\'oria, ES, 29075-910 Brazil

\end{document}